\documentclass[a4paper,12pt,reqno]{amsart}
\usepackage{amsfonts}
\usepackage{amsmath}
\usepackage{amssymb}
\usepackage[a4paper]{geometry}
\usepackage{mathrsfs}
\usepackage{csquotes}
\usepackage[colorlinks]{hyperref}
\renewcommand\eqref[1]{(\ref{#1})} 
\numberwithin{equation}{section}
\theoremstyle{plain}
\newtheorem{thm}{Theorem}[section]

\newtheorem{cor}[thm]{Corollary}

 \newtheorem{exa}[thm]{Example}
\theoremstyle{definition}

\newtheorem{rem}[thm]{Remark}

\newcommand{\Rn}{\mathbb R^{n}}

\begin{document}

   \title[Caffarelli-Kohn-Nirenberg and Sobolev type inequalities]
   {Caffarelli-Kohn-Nirenberg and Sobolev type inequalities on stratified Lie groups}

\author[M. Ruzhansky]{Michael Ruzhansky}
\address{
  Michael Ruzhansky:
  \endgraf
  Department of Mathematics
  \endgraf
  Imperial College London
  \endgraf
  180 Queen's Gate, London SW7 2AZ
  \endgraf
  United Kingdom
  \endgraf
  {\it E-mail address} {\rm m.ruzhansky@imperial.ac.uk}
  }
\author[D. Suragan]{Durvudkhan Suragan}
\address{
  Durvudkhan Suragan:
  \endgraf
  Institute of Mathematics and Mathematical Modelling
  \endgraf
  125 Pushkin str.
  \endgraf
  050010 Almaty
  \endgraf
  Kazakhstan
  \endgraf
   and
  \endgraf
  RUDN University
  \endgraf
  6 Miklukho-Maklay str., Moscow 117198
  \endgraf
  Russia
  \endgraf
  {\it E-mail address} {\rm suragan@math.kz}
  }
\author[N. Yessirkegenov]{Nurgissa Yessirkegenov}
\address{
  Nurgissa Yessirkegenov:
  \endgraf
  Institute of Mathematics and Mathematical Modelling
  \endgraf
  125 Pushkin str.
  \endgraf
  050010 Almaty
  \endgraf
  Kazakhstan
  \endgraf
  and
  \endgraf
  Department of Mathematics
  \endgraf
  Imperial College London
  \endgraf
  180 Queen's Gate, London SW7 2AZ
  \endgraf
  United Kingdom
  \endgraf
  {\it E-mail address} {\rm n.yessirkegenov15@imperial.ac.uk}
  }

\thanks{The authors were supported in parts by the EPSRC
 grant EP/R003025/1 and by the Leverhulme Grant RPG-2017-151,
 as well as by the MESRK grant 5127/GF4. The second author was also supported by the Ministry of Science of the Russian Federation (the Agreement number No 02.a03.21.0008). The third author was also supported by the MESRK grant 0825/GF4. No new data was collected or
generated during the course of research.}


     \begin{abstract}
    In this short paper, we establish a range of Caffarelli-Kohn-Nirenberg and weighted $L^{p}$-Sobolev type inequalities on stratified Lie groups.
    All inequalities are obtained with sharp constants.
    Moreover, the equivalence of the Sobolev type inequality and Hardy inequality is shown in the $L^{2}$-case.

     \end{abstract}
     \maketitle

\section{Introduction}
\label{SEC:intro}

Let us start by recalling the classical Caffarelli-Kohn-Nirenberg inequality \cite{CKN84}:

\begin{thm}\label{clas_CKN}
Let $n\in\mathbb{N}$ and let $p$, $q$, $r$, $a$, $b$, $d$, $\delta\in \mathbb{R}$ be such that $p,q\geq1$, $r>0$, $0\leq\delta\leq1$, and
\begin{equation}\label{clas_CKN0}
\frac{1}{p}+\frac{a}{n},\, \frac{1}{q}+\frac{b}{n},\, \frac{1}{r}+\frac{c}{n}>0,
\end{equation}
where
\begin{equation}\label{clas_CKN01}
c=\delta d + (1-\delta) b.
\end{equation}
Then there exists a positive constant $C$ such that
\begin{equation}\label{clas_CKN1}
\||x|^{c}f\|_{L^{r}(\Rn)}\leq C \||x|^{a}|\nabla f|\|^{\delta}_{L^{p}(\Rn)} \||x|^{b}f\|^{1-\delta}_{L^{q}(\Rn)}
\end{equation}
holds for all $f\in C_{0}^{\infty}(\Rn)$, if and only if the following conditions hold:
\begin{equation}\label{clas_CKN2}
\frac{1}{r}+\frac{c}{n}=\delta \left(\frac{1}{p}+\frac{a-1}{n}\right)+(1-\delta)\left(\frac{1}{q}+\frac{b}{n}\right),
\end{equation}
\begin{equation}\label{clas_CKN3}
a-d\geq 0 \quad {\rm if} \quad \delta>0,
\end{equation}
\begin{equation}\label{clas_CKN4}
a-d\leq 1 \quad {\rm if} \quad \delta>0 \quad {\rm and} \quad \frac{1}{r}+\frac{c}{n}=\frac{1}{p}+\frac{a-1}{n}.
\end{equation}
\end{thm}
Since the paper \cite{CKN84} the subject of such inequalities has been actively investigated. See, for instance, \cite{CW-2001} for sharpness results, \cite{NDD12} for radially symmetric functions and \cite{ZhHD15} for the generalised Baouendi-Grushin vector field. We also refer to \cite{Han15} for the Heisenberg group results that go back to the paper of Garofalo and Lanconelli \cite{GL} (see also \cite{DAmbrosio-Difur04} and references therein), in which Hardy-type inequalities on the Heisenberg group were presented.  For a short review in this direction and some further discussions we refer to recent papers \cite{ORS16}, \cite{RS17a} and \cite{RSY17} as well as to references therein.
We can also refer to \cite{CPDT07} for discussions related to the Heisenberg group.
For a recent review of some results concerning Hardy inequality on stratified groups we can refer to \cite{CCR15}, and also to \cite{CRS07}.
Some results on Hardy inequalities on homogeneous groups have appeared in
\cite{RS-identities}
and on Caffarelli-Kohn-Nirenberg inequalities in \cite{RSY17}

Since we are also interested in Sobolev inequality, let us recall it briefly. If $1<p,p^{*}<\infty$ and
\begin{equation}
\frac{1}{p}=\frac{1}{p^{*}}-\frac{1}{n},
\end{equation}
then the (Euclidean) Sobolev inequality has the form
\begin{equation}\label{EQ:Sob}
\|g\|_{L^{p}(\mathbb{R}^{n})}\leq C(p)\|\nabla g\|_{L^{p^{*}}(\mathbb{R}^{n})},
\end{equation}
where $\nabla$ is the standard gradient in $\mathbb{R}^{n}$.

The following Sobolev type inequality with respect to the operator $x\cdot \nabla$ instead of $\nabla$ has been considered in \cite{OS09}:
\begin{equation}\label{Sob}
\|g\|_{L^{p}(\mathbb{R}^{n})}\leq C'(p)\|x\cdot\nabla g\|_{L^{q}(\mathbb{R}^{n})}.
\end{equation}
For any $\lambda>0$, by substituting $g(x)=h(\lambda x)$ into \eqref{Sob}, one easily observes that $p=q$ is a necessary condition to have \eqref{Sob}.

In this paper we are interested in these inequalities in the setting of general stratified groups (or homogeneous Carnot groups). Such groups have been historically investigated by Folland \cite{F75}, with numerous subsequent contributions by many people. This class includes the Heisenberg group as the main example, as well as more general $H$-type (see e.g. \cite{GRS17}) and other groups.

The Sobolev inequality \eqref{EQ:Sob} is well known on stratified Lie groups (\cite{F75}) and, in fact, even on general graded Lie groups, see e.g. \cite[Theorem 4.4.28]{FR16}. So, here we are more interested in the Sobolev type inequalities \eqref{Sob}.
The Cafarelli-Kohn-Nirenberg inequalities on the stratified groups have been also recently investigated in \cite{RS17a} but only in the case of $p=q=r$. Here we extend it to a more general range of $p,q$ and $r$.

For the convenience of the reader we summarise briefly the results of this paper: Let $\mathbb{G}$ be a stratified group with $N$ being the dimension of the first stratum and let $|\cdot|$ be the Euclidean norm on $\mathbb{R}^{N}$.
We denote by $x'$ the variables from the first stratum of $\mathbb G$. Then we have
\begin{itemize}
\item ({\bf Sobolev type inequality}) Let $\alpha\in\mathbb{R}$. Then for any $f\in C_{0}^{\infty}(\mathbb{G}\backslash\{x'=0\})$, and all $1<p<\infty$, we have
$$
\frac{|N-\alpha p|}{p}\left\|\frac{f}{|x'|^{\alpha}}\right\|_{L^{p}(\mathbb{G})}\leq \left\|\frac{x'\cdot \nabla_{H}f}{|x'|^{\alpha}}\right\|_{L^{p}(\mathbb{G})},
$$
where $\nabla_{H}$ is the horizontal gradient on $\mathbb G$ and
$|\cdot|$ is the Euclidean norm on $\mathbb{R}^{N}$. If $N\neq\alpha p$, then the constant $\frac{|N-\alpha p|}{p}$ is sharp.

\item ({\bf Equivalence of Hardy and Sobolev type inequalities in $L^2(\mathbb G)$})
Let $N\geq3$. Then the following two statements are equivalent:

\noindent
a) For any $f\in C_{0}^{\infty}(\mathbb{G}\backslash\{x'=0\})$, we have
$$
\|f\|_{L^{2}(\mathbb{G})}\leq \frac{2}{N}\|x'\cdot \nabla_{H}f\|_{L^{2}(\mathbb{G})}.
$$
b) For any $g\in C_{0}^{\infty}(\mathbb{G}\backslash\{x'=0\})$, we have
$$
\left\|\frac{g}{|x'|}\right\|_{L^{2}(\mathbb{G})}\leq \frac{2}{N-2}\left\|\frac{x'}{|x'|}\cdot \nabla_{H}g\right\|_{L^{2}(\mathbb{G})}.
$$

\item ({\bf Caffarelli-Kohn-Nirenberg inequalities})
Let $N\neq p(1-a)$. Let $1<p,q<\infty$, $0<r<\infty$ with $p+q\geq r$ and $\delta\in[0,1]\cap\left[\frac{r-q}{r},\frac{p}{r}\right]$ and $a$, $b$, $c\in\mathbb{R}$. Assume that $\frac{\delta r}{p}+\frac{(1-\delta)r}{q}=1$ and $c=\delta(a-1)+b(1-\delta)$. Then we have the following Caffarelli-Kohn-Nirenberg type inequality for all $f\in C_{0}^{\infty}(\mathbb{G}\backslash\{0\})$:
$$
\||x'|^{c}f\|_{L^{r}(\mathbb{G})}
\leq \left|\frac{p}{N+p(a-1)}\right|^{\delta} \left\||x'|^{a}\nabla_{H}f\right\|^{\delta}_{L^{p}(\mathbb{G})}
\left\||x'|^{b}f\right\|^{1-\delta}_{L^{q}(\mathbb{G})}.
$$
The constant $\left|\frac{p}{N+p(a-1)}\right|^{\delta}$ is sharp for $p=q$ with $a-b=1$ or $p\neq q$ with $p(1-a)+bq\neq0$, or $\delta=0,1$.
\end{itemize}

Note  that these inequalities with weights from the first stratum of $\mathbb G$ also give new insights (proofs) in the Euclidean setting. By using this idea in the paper \cite{RS17a} a new general inequality (see \cite[Proposition 3.2]{RS17a}) was obtained in the Euclidean case, in particular, whose proof gave a new (simple) proof of the Badiale-Tarantello conjecture. Obtained versions may also have applications to linear and nonlinear equations of mathematical physics (see, e.g. \cite{BadTar:ARMA-2002}). In addition, to the best of our knowledge, the Caffarelli-Kohn-Nirenberg inequalities above on a (general) stratified group $\mathbb G$ are new even in the Abelian case, that is, in the Euclidean case these extend the classical Caffarelli-Kohn-Nirenberg inequalities with respect to ranges of parameters, see Example \ref{CKN_example}. Hardy inequalities for different operators is a topic with many investigations. For example, for sub-Laplacians with lower regularity, see \cite{Kogoj-Sonner:Hardy-lambda-CVEE-2016}.

In Section \ref{SEC:prelim} we briefly recall the main concepts of stratified groups and fix the notation. In Section \ref{SEC:Hardy-Sobolev} the $L^{p}$-weighted Sobolev type inequality and its equivalence to the Hardy inequality in $L^{2}$ are proved. Finally, in Section \ref{SEC:CKN} we obtain the Caffarelli-Kohn-Nirenberg type inequalities on stratified Lie groups.

\section{Preliminaries}
\label{SEC:prelim}

In this section we very briefly recall the necessary notation concerning the setting of stratified groups.

A Lie group $\mathbb{G}=(\mathbb{R}^{n}, \circ)$ is called a stratified group (or a homogeneous Carnot group) if it satisfies the following two conditions:

\begin{itemize}
\item For some natural numbers $N+N_{2}+\ldots+N_{r}=n$, that is $N=N_{1}$, the decomposition $\mathbb{R}^{n}=\mathbb{R}^{N}\times\ldots\times\mathbb{R}^{N_{r}}$ is valid, and for every $\lambda>0$ the dilation $\delta_{\lambda}:\mathbb{R}^{n}\rightarrow\mathbb{R}^{n}$ given by
    $$\delta_{\lambda}(x)=\delta_{\lambda}(x',x^{(2)},\ldots,x^{(r)}):=
    (\lambda x', \lambda^{2}x^{(2)},\ldots,\lambda^{r}x^{(r)})$$ is an automorphism of the group $\mathbb{G}$. Here $x'\equiv x^{(1)}\in\mathbb{R}^{N}$ and $x^{(k)}\in\mathbb{R}^{N_{k}}$ for $k=2,\ldots,r$.
\item Let $N$ be as in above and let $X_{1}, \ldots, X_{N}$ be the left invariant vector fields on $\mathbb{G}$ such that $X_{k}(0)=\frac{\partial}{\partial x_{k}}|_{0}$ for $k=1, \ldots, N$. Then
    $${\rm rank}({\rm Lie}\{X_{1}, \ldots, X_{N}\})=n,$$
for every $x\in\mathbb{R}^{n}$, i.e. the iterated commutators of $X_{1}, \ldots, X_{N}$ span the Lie algebra of $\mathbb{G}$.
\end{itemize}
That is, we say that the triple $\mathbb{G}=(\mathbb{R}^{n}, \circ, \delta_{\lambda})$ is a stratified group. Such groups have been thoroughly investigated by Folland \cite{F75}. A more general approach without identifying them with $\mathbb R^n$ is possible but then it can be shown to reduce to the definition above, so we may work with it from the beginning.
We refer to e.g. \cite{FR16} for more detailed discussions from the Lie algebra point of view.

Here the left invariant vector fields $X_{1}, \ldots, X_{N}$ are called the (Jacobian) generators of $\mathbb{G}$ and $r$ is called a step of $\mathbb{G}$. The number
$$Q=\sum_{k=1}^{r}kN_{k}, \;\; N_{1}=N,$$
is called the homogeneous dimension of $\mathbb{G}$. We also recall that the standard Lebesgue measure $dx$ on $\mathbb{R}^{n}$ is the Haar measure for $\mathbb{G}$ (see, e.g. \cite[Proposition 1.6.6]{FR16}).
For more details on stratified groups we refer to \cite{BLU07} or \cite{FR16}.

 The left invariant vector fields $X_{j}$ have an explicit form and satisfy the divergence theorem, see e.g. \cite{RS17b} for the derivation of the exact formula: more precisely, we can write
\begin{equation}\label{Xk}
X_{k}=\frac{\partial}{\partial x'_{k}}+\sum_{\ell=2}^{r}\sum_{m=1}^{N_{1}}a_{k,m}^{(\ell)}
(x', \ldots, x^{\ell-1})\frac{\partial}{\partial x_{m}^{(\ell)}},
\end{equation}
see also \cite[Section 3.1.5]{FR16} for a general presentation. Throughout this paper, we will also use the following notations:
$$\nabla_{H}:=(X_{1}, \ldots, X_{N})$$
for the horizontal gradient,
$${\rm div}_{H}\upsilon:=\nabla_{H}\cdot \upsilon$$
for the horizontal divergence, and
$$|x'|=\sqrt{x_{1}^{'2}+\ldots+x_{N}^{'2}}$$
for the Euclidean norm on $\mathbb{R}^{N}$.

The explicit representation of the left invariant vector fields $X_{j}$ \eqref{Xk} allows us to verify the identities
\begin{equation}\label{formula1}
|\nabla_{H}|x'|^{\gamma}|=\gamma|x'|^{\gamma-1},
\end{equation}
and
\begin{equation}\label{formula2}
{\rm div}_{H}\left(\frac{x'}{|x'|^{\gamma}}\right)=\frac{\sum_{j=1}^{N}|x'|^{\gamma}X_{j}x'_{j}-
\sum_{j=1}^{N}x'_{j}\gamma|x'|^{\gamma-1}X_{j}|x'|}{|x'|^{2\gamma}}=\frac{N-\gamma}{|x'|^{\gamma}}
\end{equation}
for any $\gamma \in \mathbb{R}$, $|x'|\neq0$.

\section{Hardy and Sobolev type inequalities on stratified Lie groups}
\label{SEC:Hardy-Sobolev}

In this section we investigate $L^{p}$-weighted Sobolev type inequality and show its equivalence to the Hardy inequality in $L^{2}$.

\begin{thm}\label{Sobolev}
Let $\mathbb{G}$ be a stratified group with $N$ being the dimension of the first stratum, and let $\alpha\in\mathbb{R}$. Then for any $f\in C_{0}^{\infty}(\mathbb{G}\backslash\{x'=0\})$, and all $1<p<\infty$, we have
\begin{equation}\label{Sobolev1}
\frac{|N-\alpha p|}{p}\left\|\frac{f}{|x'|^{\alpha}}\right\|_{L^{p}(\mathbb{G})}\leq \left\|\frac{x'\cdot \nabla_{H}f}{|x'|^{\alpha}}\right\|_{L^{p}(\mathbb{G})},
\end{equation}
where $|\cdot|$ is the Euclidean norm on $\mathbb{R}^{N}$. If $N\neq\alpha p$, then the constant $\frac{|N-\alpha p|}{p}$ is sharp.
\end{thm}

\begin{rem} In the abelian case $\mathbb{G}=(\mathbb{R}^{n},+)$, we have $N=n$, $\nabla_{H}=\nabla=(\partial_{x_{1}}, \ldots, \partial_{x_{n}})$, so \eqref{Sobolev1} implies the $L^{p}$-weighted Sobolev type inequality (see \cite{OS09}) for $\mathbb{G}=\mathbb{R}^{n}$ with the sharp constant:
\begin{equation}\label{Sobolev_euclid}
\frac{|n-\alpha p|}{p}\left\|\frac{f}{|x|^{\alpha}}\right\|_{L^{p}(\mathbb{R}^{n})}\leq \left\|\frac{x\cdot \nabla f}{|x|^{\alpha}}\right\|_{L^{p}(\mathbb{R}^{n})},
\end{equation}
for all $f\in C_{0}^{\infty}(\mathbb{R}^{n}\backslash\{0\})$, and $|x|=\sqrt{x_{1}^{2}+\ldots+x_{n}^{2}}$.
\end{rem}

The proof is a simple argument, giving the result also related to \cite[Theorem 2.12]{DAmbrosio-Difur04}.

\begin{proof}[Proof of Theorem \ref{Sobolev}]
We may assume that $\alpha p\neq N$ since for $\alpha p=N$ the inequality \eqref{Sobolev1} is trivial. By using the identity \eqref{formula2} and the divergence theorem one calculates
$$\int_{\mathbb{G}}\frac{|f(x)|^{p}}{|x'|^{\alpha p}}=\frac{1}{N-\alpha p}\int_{\mathbb{G}}|f(x)|^{p}{\rm div}_{H}\left(\frac{x'}{|x'|^{\alpha p}}\right)dx$$
$$=-\frac{p}{N-\alpha p}{\rm Re} \int_{\mathbb{G}}pf(x)|f(x)|^{p-2}\frac{\overline{x'\cdot\nabla_{H}f}}{|x'|^{\alpha p}}dx$$
$$\leq\left|\frac{p}{N-\alpha p}\right|\int_{\mathbb{G}}\frac{|f(x)|^{p-1}}{|x'|^{\alpha p}}|x'\cdot\nabla_{H}f|dx$$
$$\leq\left|\frac{p}{N-\alpha p}\right|\int_{\mathbb{G}}\frac{|f(x)|^{p-1}}{|x'|^{\alpha(p-1)}}\frac{|x'\cdot\nabla_{H}f|}{|x'|^{\alpha}}dx$$
$$\leq\left|\frac{p}{N-\alpha p}\right|\left(\frac{|f(x)|^{p}}{|x'|^{\alpha p}}dx\right)^{\frac{p-1}{p}}
\left(\frac{|x'\cdot\nabla_{H}f|^{p}}{|x'|^{\alpha p}}dx\right)^{\frac{1}{p}},$$
which implies \eqref{Sobolev1}. Here we have used H\"older's inequality in the last line. Now let us prove the sharpness of the constant. We note that the function
$$h_{1}(x)=\frac{1}{|x'|^{\frac{|N-\alpha p|}{p}}}, \quad N\neq\alpha p,$$
satisfies the following equality condition in H\"{o}lder's inequality
$$\left|\frac{p}{N-\alpha p}\right|^{p}\frac{|x'\cdot\nabla_{H}h_{1}(x)|^{p}}{|x'|^{\alpha p}}=\frac{|h_{1}(x)|^{p}}{|x'|^{\alpha p}}.$$ Thus we have showed that the constant $\frac{|N-\alpha p|}{p}$ is sharp if we approximate this function by smooth compactly supported functions.
\end{proof}
Using Schwarz's inequality in the right hand side of \eqref{Sobolev1} we see that \eqref{Sobolev1} is a refinement of the $L^{p}$-weighted Hardy inequality on stratified groups from \cite{RS17a}:

\begin{cor}\label{cor}
Let $\mathbb{G}$ be a stratified group with $N$ being the dimension of the first stratum, and let $\alpha\in\mathbb{R}$. Then for any $f\in C_{0}^{\infty}(\mathbb{G}\backslash\{x'=0\})$, and all $1<p<\infty$, we have
\begin{equation}\label{Hardy}
\frac{|N-\alpha p|}{p}\left\|\frac{f}{|x'|^{\alpha}}\right\|_{L^{p}(\mathbb{G})}\leq \left\|\frac{\nabla_{H}f}{|x'|^{\alpha-1}}\right\|_{L^{p}(\mathbb{G})},
\end{equation}
where $|\cdot|$ is the Euclidean norm on $\mathbb{R}^{N}$. If $N\neq\alpha p$ then the constant $\frac{|N-\alpha p|}{p}$ is sharp.
\end{cor}
Thus, \eqref{Sobolev1} can be regarded as a refinement of \eqref{Hardy}. The above also gives a simple proof of a part of \cite[Theorem 2.12]{DAmbrosio-Difur04}.

Now let us show the equivalence of the Sobolev type inequality and Hardy inequality on stratified groups in $L^{2}$ case:

\begin{thm}\label{Sobolev_Hardy}
Let $\mathbb{G}$ be a  stratified group with $N$ being the dimension of the first stratum with $N\geq3$. Then the following two statements are equivalent:

\begin{itemize}
\item[(a)] For any $f\in C_{0}^{\infty}(\mathbb{G}\backslash\{x'=0\})$, we have
\begin{equation}\label{Hardy_equiv}
\|f\|_{L^{2}(\mathbb{G})}\leq \frac{2}{N}\|x'\cdot \nabla_{H}f\|_{L^{2}(\mathbb{G})}.
\end{equation}
\item[(b)] For any $g\in C_{0}^{\infty}(\mathbb{G}\backslash\{x'=0\})$, we have
\begin{equation}\label{Sobolev_equiv}
\left\|\frac{g}{|x'|}\right\|_{L^{2}(\mathbb{G})}\leq \frac{2}{N-2}\left\|\frac{x'}{|x'|}\cdot \nabla_{H}g\right\|_{L^{2}(\mathbb{G})}.
\end{equation}

\end{itemize}
\end{thm}

\begin{proof} Let us first show that the statement (a) gives (b). We put $g=|x'|f$. Then one has
$$\|x'\cdot \nabla_{H}f\|_{L^{2}(\mathbb{G})}^{2}=\left\|-\frac{g}{|x'|}+\frac{x'}{|x'|}\cdot\nabla_{H}g\right\|_{L^{2}(\mathbb{G})}^{2}
$$
\begin{equation}\label{equiv1}
=\left\|\frac{g}{|x'|}\right\|_{L^{2}(\mathbb{G})}^{2}-2{\rm Re}\int_{\mathbb{G}}\overline{\frac{g(x)}{|x'|}}\frac{x'}{|x'|}\cdot \nabla_{H}g(x)dx+\left\|\frac{x'}{|x'|}\cdot\nabla_{H}g\right\|_{L^{2}(\mathbb{G})}^{2}.
\end{equation}
By \eqref{formula2}, one calculates
$$-2{\rm Re}\int_{\mathbb{G}}\overline{\frac{g(x)}{|x'|}}\frac{x'}{|x'|}\cdot \nabla_{H}g(x)dx
=-\int_{\mathbb{G}}\frac{x'}{|x'|^{2}}\nabla_{H}|g(x)|^{2}dx$$
$$=\int_{\mathbb{G}}{\rm div}_{H}\left(\frac{x'}{|x'|^{2}}\right)|g(x)|^{2}dx
=(N-2)\int_{\mathbb{G}}\frac{|g(x)|^{2}}{|x'|^{2}}dx.$$
We see from the statement (a) and \eqref{equiv1} that
$$\left\|\frac{g}{|x'|}\right\|_{L^{2}(\mathbb{G})}^{2}\leq \frac{4}{N^{2}}\left((N-1)\left\|\frac{g}{|x'|}\right\|_{L^{2}(\mathbb{G})}^{2}+\left\|\frac{x'}{|x'|}\cdot\nabla_{H}g\right\|_{L^{2}(\mathbb{G})}^{2}\right),$$
which implies \eqref{Sobolev_equiv}.

Conversely, assume that (b) holds. Put $f=g/|x'|$. Then we obtain
$$\left\|\frac{x'}{|x'|}\cdot\nabla_{H}(|x'|f)\right\|_{L^{2}(\mathbb{G})}^{2}=\|f+x'\cdot\nabla_{H}f\|_{L^{2}(\mathbb{G})}^{2}$$
$$=\|f\|_{L^{2}(\mathbb{G})}^{2}+2{\rm Re}\int_{\mathbb{G}}x'f(x)\overline{\nabla_{H}f}dx+\|x'\cdot\nabla_{H}f\|_{L^{2}(\mathbb{G})}^{2}.$$
Using \eqref{formula2}, we have
$$2{\rm Re}\int_{\mathbb{G}}x'f(x)\overline{\nabla_{H}f}dx=-N\|f\|_{L^{2}(\mathbb{G})}^{2}.$$
It follows from the statement (b) that
$$\|f\|_{L^{2}(\mathbb{G})}^{2}\leq \frac{4}{(N-2)^{2}}(\|x'\cdot\nabla_{H}f\|_{L^{2}(\mathbb{G})}^{2}-(N-1)\|f\|_{L^{2}(\mathbb{G})}^{2}),$$
which implies \eqref{Hardy_equiv}.
\end{proof}

\section{Caffarelli-Kohn-Nirenberg inequalities}\label{SEC:CKN}

In this section, we introduce new Caffarelli-Kohn-Nirenberg type inequalities in the setting of stratified groups. The proof is quite simple relying on weighted Hardy inequalities and the H\"older inequality. However, we note that already in
the Euclidean setting of $\Rn$ it also gives an extension of Theorem \ref{clas_CKN} from the point of view of indices.

\begin{thm}\label{CKN_thm}
Let $\mathbb{G}$ be a stratified group with $N$ being the dimension of the first stratum with $N\neq p(1-a)$. Let $1<p,q<\infty$, $0<r<\infty$ with $p+q\geq r$ and $\delta\in[0,1]\cap\left[\frac{r-q}{r},\frac{p}{r}\right]$ and $a$, $b$, $c\in\mathbb{R}$. Assume that
$$\frac{\delta r}{p}+\frac{(1-\delta)r}{q}=1 \; \textrm{ and } \; c=\delta(a-1)+b(1-\delta).$$
Then we have the following Caffarelli-Kohn-Nirenberg type inequality for all $f\in C_{0}^{\infty}(\mathbb{G}\backslash\{0\})$:
\begin{equation}\label{CKN_thm1}
\||x'|^{c}f\|_{L^{r}(\mathbb{G})}
\leq \left|\frac{p}{N+p(a-1)}\right|^{\delta} \left\||x'|^{a}\nabla_{H}f\right\|^{\delta}_{L^{p}(\mathbb{G})}
\left\||x'|^{b}f\right\|^{1-\delta}_{L^{q}(\mathbb{G})}.
\end{equation}
The constant in the inequality \eqref{CKN_thm1} is sharp for $p=q$ with $a-b=1$ or $p\neq q$ with $p(1-a)+bq\neq0$, or for $\delta=0,1$.
\end{thm}

\begin{rem} In the abelian case $\mathbb{G}=(\mathbb{R}^{n},+)$, we have $N=n$, $\nabla_{H}=\nabla=(\partial_{x_{1}}, \ldots, \partial_{x_{n}})$, so \eqref{CKN_thm1} implies the Caffarelli-Kohn-Nirenberg type inequality for $\mathbb{G}=\mathbb{R}^{n}$: Let $1<p,q<\infty$, $0<r<\infty$ with $p+q\geq r$ and $\delta\in[0,1]\cap\left[\frac{r-q}{r},\frac{p}{r}\right]$ and $a$, $b$, $c\in\mathbb{R}$. Assume that $\frac{\delta r}{p}+\frac{(1-\delta)r}{q}=1$ and $c=\delta(a-1)+b(1-\delta)$. Then we have
\begin{equation}\label{CKN_euclid}
\||x|^{c}f\|_{L^{r}(\Rn)}
\leq \left|\frac{p}{n+p(a-1)}\right|^{\delta} \left\||x|^{a}\nabla f\right\|^{\delta}_{L^{p}(\Rn)}
\left\||x|^{b}f\right\|^{1-\delta}_{L^{q}(\Rn)},
\end{equation}
for all $f\in C_{0}^{\infty}(\mathbb{R}^{n}\backslash\{0\})$, $|x|=\sqrt{x_{1}^{2}+\ldots+x_{n}^{2}}$ and $n\neq p(1-a)$. The constant in the inequality \eqref{CKN_euclid} is sharp for $p=q$ with $a-b=1$ or $p\neq q$ with $p(1-a)+bq\neq0$, or for $\delta=0,1$.
\end{rem}

We now indicate that the inequalities \eqref{CKN_euclid} give an extension of Theorem \ref{clas_CKN} with respect to the range of indices.

\begin{exa}\label{CKN_example} Let us take $1<p=q=r<\infty$, $a=-\frac{n-2p}{p}$, $b=-\frac{n}{p}$ and $c=-\frac{n-\delta p}{p}$.
Then by \eqref{CKN_euclid}, for all $f\in C_{0}^{\infty}(\mathbb{R}^{n}\backslash\{0\})$ we have  the inequality
\begin{equation}\label{CKN_example1}
\left\|\frac{f}{|x|^{\frac{n-\delta p}{p}}}\right\|_{L^{p}(\Rn)}\leq
\left\|\frac{\nabla f}{|x|^{\frac{n-2p}{p}}}\right\|^{\delta}_{L^{p}(\Rn)}
\left\|\frac{f}{|x|^{\frac{n}{p}}}\right\|^{1-\delta}_{L^{p}(\Rn)}, \quad 1<p<\infty, \;0\leq\delta\leq1,
\end{equation}
where $\nabla$ is the standard gradient in $\Rn$. Since we have
$$\frac{1}{q}+\frac{b}{n}=\frac{1}{p}+\frac{1}{n}\left(-\frac{n}{p}\right)=0,$$
we see that \eqref{clas_CKN0} fails, so that the inequality \eqref{CKN_example1} is not covered by Theorem \ref{clas_CKN}. Moreover, in this case, $p=q$ with $a-b=1$ hold true, so that the constant $\left|\frac{p}{n+p(a-1)}\right|^{\delta}=1$ in the inequality \eqref{CKN_example1} is sharp.
\end{exa}

\begin{proof}[Proof of Theorem \ref{CKN_thm}] {Case $\delta=0$}. In this case, we have $q=r$ and $b=c$ by $\frac{\delta r}{p}+\frac{(1-\delta)r}{q}=1$ and $c=\delta(a-1)+b(1-\delta)$, respectively. Then, the inequality \eqref{CKN_thm1} is reduces to the  trivial estimate
$$
\||x'|^{b}f\|_{L^{q}(\mathbb{G})}
\leq \left\||x'|^{b}f\right\|_{L^{q}(\mathbb{G})}.
$$
{Case $\delta=1$}. Notice that in this case, $p=r$ and $a-1=c$. By \eqref{Hardy}, we have for $N+c p\neq0$ the inequality
$$\||x'|^{c}f\|_{L^{r}(\mathbb{G})}\leq \left|\frac{p}{N+cp}\right|\||x'|^{c+1}\nabla_{H}f\|_{L^{r}(\mathbb{G})}.$$
In this case, the constant in \eqref{CKN_thm1} is sharp, since the constants in \eqref{Hardy} is sharp.

{Case $\delta\in(0,1)\cap\left[\frac{r-q}{r},\frac{p}{r}\right]$}.
Taking into account $c=\delta(a-1)+b(1-\delta)$,  a direct calculation gives $$\||x'|^{c}f\|_{L^{r}(\mathbb{G})}=
\left(\int_{\mathbb{G}}|x'|^{cr}|f(x)|^{r}dx\right)^{\frac{1}{r}}
=\left(\int_{\mathbb{G}}\frac{|f(x)|^{\delta r}}{|x'|^{\delta r (1-a)}} \frac{|f(x)|^{(1-\delta)r}}{|x'|^{-br(1-\delta)}}dx\right)^{\frac{1}{r}}.$$
Since we have $\delta\in(0,1)\cap\left[\frac{r-q}{r},\frac{p}{r}\right]$ and $p+q\geq r$, then by using H\"{o}lder's inequality for $\frac{\delta r}{p}+\frac{(1-\delta)r}{q}=1$, we obtain
$$\||x'|^{c}f\|_{L^{r}(\mathbb{G})}
\leq \left(\int_{\mathbb{G}}\frac{|f(x)|^{p}}{|x'|^{p(1-a)}}dx\right)^{\frac{\delta}{p}}
\left(\int_{\mathbb{G}}\frac{|f(x)|^{q}}{|x'|^{-bq}}dx\right)^{\frac{1-\delta}{q}}$$
\begin{equation}\label{CKN_thm1_1}=\left\|\frac{f}{|x'|^{1-a}}\right\|^{\delta}_{L^{p}(\mathbb{G})}
\left\|\frac{f}{|x'|^{-b}}\right\|^{1-\delta}_{L^{q}(\mathbb{G})}.
\end{equation}
Here we note that when $p=q$ and $a-b=1$, H\"{o}lder's equality condition is satisfied for any function. We also note that in the case $p\neq q$ the function
\begin{equation}\label{Holder_eq_2}
h_{2}(x)=|x'|^{\frac{1}{(p-q)}\left(p(1-a)+bq\right)}
\end{equation} satisfies H\"{o}lder's equality condition:
$$\frac{|h_{2}|^{p}}{|x'|^{p(1-a)}}=\frac{|h_{2}|^{q}}{|x'|^{-bq}}.$$
If $N\neq p(1-a)$, then by \eqref{Hardy}, we have
\begin{equation}\label{1}\left\|\frac{f}{|x'|^{1-a}}\right\|^{\delta}_{L^{p}(\mathbb{G})}\leq
\left|\frac{p}{N+p(a-1)}\right|^{\delta} \left\|\frac{\nabla_{H}f}{|x'|^{-a}}\right\|^{\delta}_{L^{p}(\mathbb{G})}.
\end{equation}
Putting this in \eqref{CKN_thm1_1}, one has
$$\||x'|^{c}f\|_{L^{r}(\mathbb{G})}\leq
\left|\frac{p}{N+p(a-1)}\right|^{\delta} \left\|\frac{\nabla_{H}f}{|x'|^{-a}}\right\|^{\delta}_{L^{p}(\mathbb{G})}
\left\|\frac{f}{|x'|^{-b}}\right\|^{1-\delta}_{L^{q}(\mathbb{G})}.$$
When we prove $\eqref{1}$, in exactly the same way as in the proof of Theorem \ref{Sobolev}, we note that
\begin{equation}\label{2}
h_{3}(x)=|x'|^{C},\;\;\;\;C\neq0,
\end{equation}
satisfies the H\"older equality condition. Therefore, in the case $p=q$, $a-b=1$ H\"older's equality condition of the inequalities \eqref{CKN_thm1_1} and \eqref{1} holds true for $h_{3}(x)$ in \eqref{2}. Moreover, in the case $p\neq q$, $p(1-a)+bq\neq0$ H\"older's equality condition of the inequalities \eqref{CKN_thm1_1} and \eqref{1} holds true for $h_{2}(x)$ in \eqref{Holder_eq_2}. Therefore, the constant in \eqref{CKN_thm1} is sharp when $p=q$, $a-b=1$ or $p\neq q$, $p(1-a)+bq\neq0$.
\end{proof}

\end{document}